\documentclass[12pt,oneside]{amsart}

\usepackage{amssymb}
\usepackage{amscd}

\title[Characterization of a Hermitian curve by Galois point]{Characterization of a Hermitian curve \\ by Galois point}
\author{Satoru Fukasawa}

\address{Department of Mathematical Sciences, Faculty of Science, Yamagata University, Kojirakawa-machi 1-4-12, Yamagata 990-8560, Japan}
\email{s.fukasawa@sci.kj.yamagata-u.ac.jp} 

\newtheorem{theorem}{}
 
\newtheorem{proposition}{}

\newtheorem{lemma}{Lemma}
\newtheorem{fact}{Fact}

\theoremstyle{definition}

\begin{document}
\begin{abstract} 
For a plane curve, a point in the projective plane is said to be Galois when the point projection induces a Galois extension of function fields. 
We give a new characterization of a Fermat curve whose degree minus one is a power of $p$ in characteristic $p>2$, which is sometimes called Hermitian,  by the number of Galois points lying on the curve.  
\end{abstract}
\maketitle

\section{Introduction}  
Let $C \subset \Bbb P^2$ be an irreducible plane curve of degree $d \ge 4$ over an algebraically closed field $K$ of characteristic $p \ge 0$ and let $K(C)$ be the function field of $C$. 
The point projection $\pi_P: C \dashrightarrow \Bbb P^1$ from a point $P \in \Bbb P^2$ induces a field extension $K(C)/\pi_P^*K(\Bbb P^1)$ of function fields. 
When the extension is Galois, we call the point $P$ a {\it Galois point} for $C$. 
This notion was introduced by H. Yoshihara (\cite{fukasawa1}, \cite{miura-yoshihara}, \cite{yoshihara1}). 
A Galois point $P \in \Bbb P^2$ is said to be inner if $P \in C_{\rm sm}$, where $C_{\rm sm}$ is the smooth locus of $C$. 
We denote by $\delta(C)$ the number of inner Galois points. 
If there exist infinitely many inner Galois points, we define as $\delta(C)=\infty$. 
Otherwise, we define as $\delta(C)<\infty$. 
It is remarkable that many classification results of algebraic varieties have been obtained in the theory of Galois point.

When $p=0$, Yoshihara determined the number $\delta(C)$ for smooth curves (\cite{yoshihara1}). 
Miura gave a certain inequality related to $\delta(C)$ if $d-1$ is prime (\cite{miura}). 
In $p>0$, Homma proved that $\delta(H)=(p^e)^3+1$ for a Fermat curve $H$ of degree $p^e+1$ (\cite{homma2}), which is sometimes called Hermitian. 
Recently, the present author determined $\delta(C)$ for any other smooth curve $C$ (\cite{fukasawa3}). 
On the other hand, Hasegawa and the present author \cite{fukasawa-hasegawa} found examples of plane curves having infinitely many Galois points and classified them. 
Then, we have a natural question: {\it When $\delta(C)$ is finite, what is the maximal number $\delta(C)$?} 
The purpose of this paper is to answer this question, which leads to a new characterization of a Hermitian curve. 

\begin{theorem} \label{main theorem}
Let $C \subset \Bbb P^2$ be an irreducible plane curve of degree $d \ge 4$ over an algebraically closed field $K$ of characteristic $p \ne 2$. 
If $\delta(C) < \infty$, then $\delta(C) \le (d-1)^3+1$. 
Furthermore, $\delta(C)=(d-1)^3+1$ if and only if $p>0$, $d-1$ is a power of $p$, and $C$ is projectively equivalent to a Fermat curve. 
\end{theorem}

\section{Preliminaries}
Let $(X:Y:Z)$ be a system of homogeneous coordinates of the projective plane $\Bbb P^2$ and let $C \subset \Bbb P^2$ be an irreducible plane curve of degree $d \ge 4$. 
We denote by $C_{\rm sm}$ the smooth locus of $C$ and by ${\rm Sing}(C)$ the singular locus of $C$. 
If $P \in C_{\rm sm}$, we denote by $T_PC \subset \Bbb P^2$ the (projective) tangent line at $P$. 
For a projective line $l \subset \Bbb P^2$ and a point $P \in C \cap l$, we denote by $I_P(C, l)$ the intersection multiplicity of $C$ and $l$ at $P$.  

A tangent line at a singular point $Q \in {\rm Sing}(C)$ is defined as follows. Let $f(x,y)$ be the defining polynomial of $C$ in the affine plane defined by $Z \ne 0$, and let $Q=(0:0:1)$. 
We can write $f=f_m+f_{m+1}+\dots+f_d$, where $f_i$ is the $i$-th homogeneous component.
A tangent line at $Q$ is the line defined by an irreducible component of $f_m$. 
Therefore, a line $l$ passing through $Q$ is a tangent line at $Q$ if and only if $I_Q(C,l)>m$. 

Let $r: \hat{C} \rightarrow C$ be the normalization and let $g$ (or $g_C$) be the genus of $\hat{C}$.  
We denote by $\overline{PR}$ the line passing through points $P$ and $R$ when $R \ne P$, and by $\pi_P: C \dashrightarrow \Bbb P^1; R \mapsto \overline{PR}$ the point projection from a point $P \in \Bbb P^2$. 
We write $\hat{\pi}_P=\pi_P \circ r$. 
We denote by $e_{\hat{R}}$ the ramification index of $\hat{\pi}_P$ at $\hat{R} \in \hat{C}$. 
If $R=r(\hat{R}) \in C_{\rm sm}$, then we denote $e_{\hat{R}}$ also by $e_R$.   
It is not difficult to check the following.  

\begin{lemma} \label{index}
Let $P \in \Bbb P^2$ and let $\hat{R} \in \hat{C}$ with $r(\hat{R})=R \ne P$. 
Then for $\hat{\pi}_P$ we have the following. 
\begin{itemize}
\item[(1)] If $P \in C_{\rm sm}$, then $e_P=I_P(C, T_PC)-1$.  
\item[(2)] If $h$ be a linear polynomial defining $\overline{RP}$, then $e_{\hat{R}}={\rm ord}_{\hat{R}}r^*h$. 
In particular, if $R$ is smooth, then $e_R =I_R(C, \overline{PR})$.  
\end{itemize} 
\end{lemma}

Let $\check{\Bbb P}^2$ be the dual projective plane, which parameterizes lines on $\Bbb P^2$. 
The dual map $\gamma: C_{\rm sm} \rightarrow \check{\Bbb P}^2$ of $C$ is a rational map which assigns a smooth point $P \in C_{\rm sm}$ to the tangent line $T_PC \in \check{\Bbb P}^2$ at $P$, and the dual curve $C^* \subset \check{\Bbb P}^2$ is the closure of the image of $\gamma$. 
We denote by $s(\gamma)$ the separable degree of the field extension $K(C)/\gamma^*K(C^*)$, which is induced from the dual map $\gamma$ of $C$ onto $C^*$, by $q(\gamma)$ the inseparable degree, and by $M(C)$ the generic order of contact (i.e. $I_P(C, T_PC)=M(C)$ for a general point $P \in C$), throughout this paper. 
If the dual map $\gamma$ is separable onto $C^*$, then $s(\gamma)=1$ and $M(C)=2$ (see, for example, \cite[Proposition 1.5]{pardini}). 
If the dual map $\gamma$ of $C$ is not separable, then it follows from a theorem of Hefez-Kleiman (\cite[(3.4)]{hefez-kleiman}) that $M(C)=q(\gamma)$. 
Using this theorem and B\'{e}zout's theorem, we find that $d \ge s(\gamma)q(\gamma)$. 

The order sequence of the morphism $r:\hat{C} \rightarrow \Bbb P^2$ is $\{0, 1, M(C)\}$ (see \cite[Ch. 7]{hkt}, \cite{stohr-voloch}). 
If $\hat{R} \in \hat{C}$ is a non-singular branch, i.e. there exists a line defined by $h=0$ with ${\rm ord}_{\hat{R}}r^*h=1$, then there exists a unique tangent line at $R=r(\hat{R})$ defined by $h_{\hat{R}}=0$ such that ${\rm ord}_{\hat{R}}r^*h_{\hat{R}} \ge M(C)$.  
We denote by $T_{\hat{R}}C \subset \Bbb P^2$ this tangent line, and by $\nu_{\hat{R}}$ the order ${\rm ord}_{\hat{R}}r^*h_{\hat{R}}$ of the tangent line $h_{\hat{R}}=0$ at $\hat{R}$. 
If $\nu_{\hat{R}}-M(C)>0$, then we call the point $\hat{R}$ (or $R=r(\hat{R})$ if $R \in C_{\rm sm}$) a {\it flex}.  
We denote by $\hat{C}_0 \subset \hat{C}$ the set of all non-singular branches and by $F(\hat{C}) \subset \hat{C}_0$ the set of all flexes.  
We recall the following (see \cite[Theorem 1.5]{stohr-voloch}).  

\begin{fact}[Count of flexes] \label{fexes} We have
$$ \sum_{\hat{R} \in \hat{C}_0} (\nu_{\hat{R}}-M(C)) \le (M(C)+1)(2g-2)+3d. $$
\end{fact}

We also recall Pl\"ucker formula. (This version is obtained easily from considering the projection from a general point of $\Bbb P^2$ and the number of singular points of $C^*$, since a general point of $\Bbb P^2$ corresponds to a general line in $\check{\Bbb P}^2$. See also \cite{piene}.)

\begin{fact}[Pl\"ucker formula] \label{plucker} 
Let $d^*$ be the degree of the dual curve $C^*$.
Then, 
$$ d^* \le 2g-2+2d. $$
If $s(\gamma)=1$, then the number of multiple tangent lines (i.e. lines $L$ such that there exist two distinct points $\hat{R}_1, \hat{R}_2 \in \hat{C}_0$ with $L=T_{\hat{R}_1}C=T_{\hat{R}_2}C$) is at most 
$$ \frac{(d^*-1)(d^*-2)}{2} \le \frac{(2g-2+2d-1)(2g-2+2d-2)}{2}. $$
\end{fact}

We recall the definition of strangeness. 
If there exists a point $Q \in \Bbb P^2$ such that almost all tangent lines of $C$ pass through $Q$, then $C$ is said to be {\it strange} and $Q$ is called a strange center (see \cite{bayer-hefez}, \cite{kleiman}).  
It is easily checked that a strange center is unique for a strange curve. 
Using Lemma \ref{index}, we find that the projection $\hat{\pi}_Q$ from a point $Q$ is not separable if and only if $C$ is strange and $Q$ is the strange center.
If $C$ is strange, then we can identify the dual map $\gamma$ with the projection $\pi_Q$ from the strange center $Q$. 

We denote by $\Delta \subset \hat{C}$ the set of all points $\hat{P} \in \hat{C}$ such that $r(\hat{P}) \in C$ is smooth and Galois with respect to a plane curve $C \subset \Bbb P^2$.  
We denote by $G_P$ the group of birational maps from $C$ to itself corresponding to the Galois group ${\rm Gal}(K(C)/\pi_P^*K(\Bbb P^1))$ when $P$ is Galois.
We find easily that the group $G_P$ is isomorphic to a subgroup of the automorphism group ${\rm Aut}(\hat{C})$ of $\hat{C}$. 
Frequently, we identify $G_P$ with the subgroup.  
If a Galois covering $\theta:C \rightarrow C'$ between smooth curves is given, then the Galois group $G$ acts on $C$ naturally. 
We denote by $G(P)$ the stabilizer subgroup of $P$ and by $e_P$ the ramification index at $P$.  
The following fact is useful to find Galois points (see \cite[III. 7.1, 7.2 and 8.2]{stichtenoth}). 
\begin{fact} \label{Galois covering} 
Let $\theta: C \rightarrow C'$ be a Galois covering of degree $d$ with a Galois group $G$. 
Then we have the following. 
\begin{itemize}
\item[(1)] The order of $G(P)$ is equal to $e_P$ at $P$ for any point $P \in C$. 
\item[(2)] If $\theta(P)=\theta(Q)$, then $e_P=e_Q$. 
\item[(3)] The index $e_P$ divides the degree $d$. 
\end{itemize}
\end{fact}

By using Lemma \ref{index} and Fact \ref{Galois covering}, we have the following. 

\begin{lemma} \label{twoGalois} 
Let $P_1, P_2 \in C_{\rm sm}$ be two distinct Galois points and let $h$ be a defining polynomial of the line $\overline{P_1P_2}$. 
Then, ${\rm ord}_{\hat{R}} r^*h=1$ for any $\hat{R} \in \hat{C}$ with $R=r(\hat{R}) \in \overline{P_1P_2}$ (maybe $R=P_1$ or $P_2$). 
\end{lemma}

\begin{proof} 
Assume that $R=P_2$ and ${\rm ord}_{\hat{R}}r^*h \ge 2$. 
It follows from Lemma \ref{index} that $e_{P_2}=I_{P_2}(C, \overline{P_1P_2})-1$ for a projection $\hat{\pi}_{P_2}$. 
Then, by Fact \ref{Galois covering}(2) and Lemma \ref{index}, $I_{P_1}(C, \overline{P_1P_2})=I_{P_2}(C, \overline{P_1P_2})-1$. 
If $I_{P_1}(C, \overline{P_1P_2}) \ge 2$, then, for $\hat{\pi}_{P_1}$, $e_{P_1}=I_{P_1}(C, \overline{P_1P_2})-1 \ge 1$ and $e_{P_1}=e_{P_2}=I_{P_2}(C, \overline{P_1P_2})$, by Lemma \ref{index} and Fact \ref{Galois covering}(2). 
Then, we have $I_{P_2}(C, \overline{P_1P_2})-2=I_{P_2}(C, \overline{P_1P_2})$. 
This is a contradiction. 
Therefore, $I_{P_1}(C, \overline{P_1P_2})=1$ and $I_{P_2}(C, \overline{P_1P_2})=2$. 
Since $d >3$, there exist a point $\hat{R}_0 \in \hat{C}$ such that $R_0=r(\hat{R}_0) \in \overline{P_1P_2}$ and $R_0 \ne P_1, P_2$. 
By Fact \ref{Galois covering}(2), we have $e_{\hat{R}_0}=2$ for $\hat{\pi}_{P_1}$ and $e_{\hat{R}_0}=1$ for $\hat{\pi}_{P_2}$. 
This is a contradiction, because $e_{\hat{R}}={\rm ord}_{\hat{R}_0}r^*h$ for each case, by Lemma \ref{index}. 

Assume that $R \ne P_1, P_2$ and ${\rm ord}_{\hat{R}}r^*h \ge 2$. 
Then, by considering $\hat{\pi}_{P_1}$ and Lemma \ref{index} and Fact \ref{Galois covering}(2), $I_{P_2}(C, \overline{P_1P_2})={\rm ord}_{\hat{R}}r^*h$. 
Then, by considering $\hat{\pi}_{P_2}$ and Lemma \ref{index} and Fact \ref{Galois covering}(2), $I_{P_2}(C, \overline{P_1P_2})-1={\rm ord}_{\hat{R}}r^*h$. 
We have $I_{P_2}(C, \overline{P_1P_2})=I_{P_2}(C, \overline{P_1P_2})-1$. 
This is a contradication.  
\end{proof} 

Finally in this section, we mention properties of Galois covering between rational curves.  

\begin{lemma} \label{rational-ramification}
Let $\theta: \Bbb P^1 \rightarrow \Bbb P^1$ be a Galois covering of degree $d \ge 3$ with a Galois group $G$. 
Then we have the following. 
\begin{itemize} 
\item[(1)] Any automorphism $\sigma \in G \setminus \{1\}$ fixes some point. 
\item[(2)] If $P$ is a ramification point and the fiber $\theta^{-1}(\theta(P))$ consists of at least two points, then there exists a ramification point $P'$ with $\theta(P') \ne \theta(P)$.  
\item[(3)] If $\theta$ is ramified only at $P$ and $e_P=d$, then $p>0$ and $d$ is a power of $p$. 
\item[(4)] If $p>0$, $d$ is a power of $p$ and the index $e_P$ at a point $P$ is equal to $d$, then $P$ is a unique ramification point. 
If $P=(1:0)$, then any element $\sigma \in G$ is represented by a matrix 
$$ A_{\sigma}=\left(\begin{array}{cc}
1 & a(\sigma) \\
0 & 1 
\end{array}\right) $$
for some $a(\sigma) \in K$. 
Furthermore, the set $\{a(\sigma)| \sigma \in G\} \subset K$ forms an additive subgroup. 
\end{itemize}  
\end{lemma}

\begin{proof}
Note that any automorphism of $\Bbb P^1$ is represented by a matrix $A_{\sigma}$. 
This implies the assertion (1). 

We consider the assertion (2). 
We assume that $\theta^{-1}(\theta(P))=\{P_1, P_2, \ldots, P_s\}$ with $s \ge 2$. 
Then, it follows from Fact \ref{Galois covering}(1)(2) that the order of $G(P_i)$ is equal to the one of $G(P_j)$ for any $i,j$.  
Since $G(P_i) \cap G(P_j)$ is not empty as a set, $\bigcup_iG(P_i) \ne G$ as a set by considering the order. 
Let $\tau \in G \setminus(\bigcup_iG(P_i))$.
By the assertion (1), there is a fixed point of $P'$ by $\tau$.  
It follows from Fact \ref{Galois covering}(1) that $P'$ is a ramification point of $\theta$. 
Then, $\theta(P') \ne \theta(P)$ because $\tau \not\in \bigcup_iG(P_i)$. 

We consider the assertion (3). 
We may assume that $P=(1:0)$. 
Let $e_P=ql$, where $q$ is a power of $p$ and $l$ is not divisible by $p$. 
Let $\sigma \in G(P)$ be any element. 
Then, $\sigma$ is represented by a matrix 
$$A_{\sigma}=\left(\begin{array}{cc}
\zeta & a \\
0 & 1 
\end{array}\right)
$$
as an automorphism of $\Bbb P^1$,
where $\zeta$ is an $l$-th root of unity and $a \in K$,  
because $\sigma(P)=P$ and $\sigma^{ql}=1$. 
If $\zeta \ne 1$, then we find that $\sigma$ has two fixed points, by direct computations. 
Therefore, $\zeta=1$ by our assumption. 
Then, any element of $G(P) \setminus \{1\}$ is of order $p$, by direct computations.  
If $l>1$ then there exists an element whose order is not divisible by $p$, by Sylow's theorem. 
This is a contradiction.
Therefore, $l=1$.

We consider the assertion (4). 
We may assume that $P=(1:0)$. 
Let $e_P=q$, where $q$ is a power of $p$. 
It follows from Fact \ref{Galois covering}(1) that the order of $G(P)$ is equal to $q$. 
Let $\sigma \in G(P)$. 
Then, by direct computations, $\sigma$ is represented by a matrix 
\begin{align*} 
A_{\sigma}=\left(\begin{array}{cc}
1 & a(\sigma) \\
0 & 1 
\end{array}\right) 
\end{align*} 
as an automorphism of $\Bbb P^1$, where $a(\sigma) \in K$,  
because $\sigma(P)=P$ and $\sigma^{q}=1$. 
Then, it is not difficult to check that $\sigma$ fixes only a point $P$ and the subset $\{a(\sigma)|\sigma \in G\} \subset K$ forms an additive subgroup. 
\end{proof}

\section{Proof} 
If $p>0$, $d-1$ is a power of $p$ and $C$ is projectively equivalent to a Fermat curve of degree $d$, then it follows from a result of Homma \cite{homma2} that $\delta(C)=(d-1)^3+1$.  

Throughout this section, we assume that $(d-1)^3+1=(2\times ((d-1)(d-2)/2)-2)d+3d \le \delta(C) < \infty$. 
Let $\lambda(C)$ be the cardinality of $\Delta \setminus F(\hat{C})$ and let $\mu(C)=(d-1)^3+1-\{(2g-2)(M(C)+1)+3d \}$. 
Then, $\lambda(C) \ge \mu(C)$, and $\mu(C)>0$ if $g<(d-1)(d-2)/2$ or $M(C)< d-1$. 
Since the present author proved that $\delta(C)=0$ or $\infty$ if $d=M(C)$ in \cite{fukasawa2}, we may assume that $d>M(C)$. 
It follows from a result of the author \cite{fukasawa3} for smooth curves (or a generalization of Pardini's theorem by Hefez \cite{hefez} and Homma \cite{homma1}) that $\mu(C)=0$ only if $p>0$, $d-1$ is a power of $p$, and $C$ is a Fermat curve of degree $d$.
Therefore, we may assume that $\mu(C) > 0$. 

(I) {\it The case where there exists a singular point $Q$ with multiplicity $d-1$.} 
Then, $\hat{C}$ is rational and $Q$ is a unique singular point. 
It follows from B\'ezout's theorem that the tangent line $T_PC$ at any smooth point $P$ does not contain $Q$. 
Since $d>M(C)$, $T_PC$ intersects some smooth point $R$. 

(I-1) Assume that $M(C) \ge 3$. 
Since $\lambda(C)>0$, there exists a smooth point $P$ which is Galois and $I_P(C, T_PC)=M(C)$. 
Then, it follows from Lemma \ref{index} and Fact \ref{Galois covering}(2) that $I_R(C, T_RC)=M(C)-1 \ge 2$. 
This is a contradiction to the order sequence $\{0, 1, M(C)\}$. 

(I-2) Assume that $M(C)=2$. 
Note that for any $R \in C_{\rm sm}$, $T_RC$ contains at most one inner Galois point by Lemma \ref{twoGalois}. 
Therefore, we have at least $\mu(C)$ Galois points $P$ which do not lie on any tangent lines at flexes. 

(I-2-1) Assume that $s(\gamma)=1$. 
Let $P \in r(\Delta) \setminus \bigcup_{\hat{R} \in F(\hat{C})}T_{\hat{R}}C$. 
Assume that the fiber $r^{-1}(Q)$ contains two or more points. 
It follows from Lemma \ref{rational-ramification}(2) that there exists a ramification point $\hat{R} \in \hat{C}$ with $R=r(\hat{R}) \ne Q$ for $\hat{\pi}_P$.  
It follows from Lemma \ref{index} and Fact \ref{Galois covering}(4) that $P \in T_RC$ and $I_{P_0}(C, T_{R}C)=2$ for any point $P_0\in C \cap T_RC$ with $P_0 \ne P$, for each $P \in r(\Delta) \setminus \bigcup_{\hat{R} \in F(\hat{C})}T_{\hat{R}}C$. 
Therefore, we have at least $\mu(C)$ multiple tangent lines.  
It follows from Fact \ref{plucker} that 
$$ \mu(C) \le \frac{(d_0-1)(d_0-2)}{2}, $$
where $d_0=2g-2+2d=2d-2$. 
Since $\mu(C)= d^3-3d^2+6$, we have an inequality 
$$ d^3-3d^2+6 \le 2d^2-7d+6. $$
Then, we have $g_1(d):=d^3-5d^2+7d \le 0$. 
Since $g_1(4)=12$, this is a contradiction.

Assume that the fiber $r^{-1}(Q)$ consists of a unique point $\hat{Q}$. 
Then, by Fact \ref{Galois covering}(2), $\hat{\pi}_P$ is ramified at $\hat{Q}$ with index $d-1$. 
It follows from Lemma \ref{rational-ramification}(3)(4) that there exists a Galois point $P$ such that $\hat{\pi}_P$ is ramified only at $\hat{Q}$ if and only if $p>0$ and $d-1$ is a power of $p$. 
If $\hat{\pi}_P$ has another ramification point for any $P \in r(\Delta) \setminus \bigcup_{\hat{R} \in F(\hat{C})}T_{\hat{R}}C$, then, similarly to the discussion above, there is a contradiction. 
Therefore, $d-1$ is a power of $p$ and $\hat{Q}$ is a unique ramification point of $\hat{\pi}_P$ for any $P \in r(\Delta) \setminus \bigcup_{\hat{R} \in F(\hat{C})}T_{\hat{R}}C$, by Lemma \ref{rational-ramification}(4) again. 
Let the normalization $r(s:t)=(\phi_0(s, t): \phi_1(s,t):\phi_2(s,t))$, where $\phi_i$ is a homogeneous polynomial of degree $d$ in variables $s, t$ for $i=0, 1, 2$. 
For a suitable system of coordinates, we may assume that $Q=(1:0:0)$ and a line $Z=0$ is a tangent line at $Q$. 
Since a solution of $\phi_2(s, t)=0$ is unique, we may assume that $\phi_2(s, t)=t^d$. 
Then, $\hat{Q}=(1:0)$. 
Since the projection $\hat{\pi}_Q$ from $Q$ is given by $(s:1) \mapsto (\phi_1(s, 1):1)$ and this is birational, $\phi_1(s, 1)$ is of degree one. 
Therefore, we may assume that $\phi_1(s, t)=st^{d-1}$. 
We may also assume that $P=(0:0:1)$. 
Then, $\phi_0(s, 1)=\sum_{i=1}^d a_i s^i$ for some $a_i \in K$. 
The projection $\hat{\pi}_P$ from $P$ is given by $(\sum_{i=1}^da_is^i:s)=(\sum_{i=1}^da_is^{i-1}:1)$. 
Since $\hat{\pi}_P$ gives a Galois covering and $\sigma(\hat{Q})=\hat{Q}$ for any $\sigma \in G_P$, it follows from Lemma \ref{rational-ramification}(4) and \cite[Proposition 1.1.5 and Theorem 1.2.1]{goss} that $a_{i}=0$ if $i-1$ is not a power of $p$. 
Let $P_2=(\phi_0(\alpha,1): \alpha:1)$ be inner Galois. 
Then, the projection $\hat{\pi}_{P_2}$ is given by $(\phi_0(s, 1)-\phi_0(\alpha): s-\alpha)$. 
Let $u:=s-\alpha$. 
Then, $\hat{\pi}_{P_2}=(\phi_0(u+\alpha, 1)-\phi_0(\alpha):u)=(\sum_{i=0}^e a_i\{(u+\alpha)^{p^i+1}-\alpha^{p^i+1}\}:u)$. 
Note that $\{(u+\alpha)^{p^i+1}-\alpha^{p^i+1}\}/u=u^{p^i}+\alpha u^{p^i-1}+\alpha^{p^i}$. 
By considering the differential of this polynomial, if $\alpha \ne 0$, then $\hat{\pi}_{P_2}$ is ramified other points than $\hat{Q}$.  
This is a contradiction to the uniqueness of the ramification point.   

(I-2-2) Assume that $s(\gamma) \ge 2$. 
Then, $q(\gamma) \ge 2$ and the number of tangent lines whose contact points are strictly less than $s(\gamma)$ is at most
$$ 2g_C-2-s(\gamma)(2g_{C^*}-2)=-2+2s(\gamma) \le -2+2(d/2)=d-2 $$
by Riemann-Hurwitz formula. 
Since $\mu(C)-(d-2)= d^3-3d^2-d+8 >0$, there exist an inner Galois point $P$ and a smooth point $R \in C_{\rm sm}$ with $R \ne P$ such that $T_PC=T_RC$ and $I_P(C, T_PC)=I_R(C, T_RC)=2$. 
By Lemma \ref{index} and Fact \ref{Galois covering}(4), this is a contradiction. 

(II) {\it The case where there exists NO singular point with multiplicity $d-1$.} 
Firstly, we prove that $\hat{C}_0 =\hat{C}$. 
Let $Q$ be a singular point with multiplicity $m \le d-2$. 
Note that the number of tangent line at $Q$ is at most $m$. 

Assume that $Q$ is not a strange center. 
We prove that any point $\hat{R} \in \hat{C}$ with $r(\hat{R})=Q$ is a non-singular branch. 
If there exists a line containg $Q$ and two Galois points, then we have this assertion by Lemma \ref{twoGalois}. 
Therefore, we consider the case where any line containing $Q$ has at most one inner Galois point. 
If we consider the projection $\hat{\pi}_{Q}$ from $Q$, then the number of ramification points is at most $2g-2+2(d-2) \le d^2-d-4$. 
Since $\delta(C) \ge (d-1)^3+1$, there exist a Galois point $P$ and a point $\hat{R} \in \hat{C}$ with $R=r(\hat{R}) \ne P, Q$ such that ${\rm ord}_{\hat{R}}r^*h=1$, where $h$ is a defining polynomial of the line $\overline{PR}$, by Lemma \ref{index}. 
It follows from Fact \ref{Galois covering}(2) that any point $\hat{R}$ in the fiber $r^{-1}(Q)$ is a non-singular branch. 

We prove that $Q$ is not a strange center. 
If there exists a line containg $Q$ and two Galois points, then we have this assertion by Lemma \ref{twoGalois}.
Therefore, we consider the case where any line containing $Q$ has at most one inner Galois point. 
Assume that $Q$ is a strange center. 
If we consider the projection $\hat{\pi}_{Q}$ from $Q$, then the number of ramification points of the separable closure of $K(C)/\hat{\pi}_Q^*K(\Bbb P^1)$ is at most $2g-2+2(d-2) \le d^2-d-4$. 
Since $\delta(C) \ge (d-1)^3+1$, there exist a Galois point $P$ such that $I_P(C, \overline{PQ})=M(C)$. 
Since any point $\hat{R}$ with $r(\hat{R}) \ne Q$ is a non-singular branch by discussions above, $Q \in T_{\hat{R}}C$ for any point $\hat{R}$ with $r(\hat{R}) \ne Q$. 
Therefore, the projection $\hat{\pi}_P$ is ramified only at points in a line $\overline{PQ}$. 
Since the ramification index at $P$ for $\hat{\pi}_P$ is equal to $M(C)-1$ by Lemma \ref{index}, there exist only tame ramification points for $\hat{\pi}_P$, by Fact \ref{Galois covering}(2). 
By Riemann-Hurwitz formula, this is a contradiction.

(II-1)
Assume that $M(C) \ge 3$. 
Since $\lambda(C)>0$, there exists a Galois point $P \in C_{\rm sm}$ such that $I_P(C, T_PC)=M(C)$. 
Since $d>M(C)$, there exists a point $R \in C \cap T_PC$. 
Then, the ramification index $e_{\hat{R}}=M(C)-1 \ge 2$ at $\hat{R}$ for $\hat{\pi}_P$, where $\hat{R} \in \hat{C}$ with $r(\hat{R})=R$. 
This is a contradiction to the order sequence $\{0, 1, M(C)\}$.  

(II-2) Assume that $M(C)=2$. 
Note that for any $\hat{R} \in \hat{C}$, $T_{\hat{R}}C$ contains at most one inner Galois points by Lemma \ref{twoGalois}.
Therefore, we have at least $\mu(C)$ Galois points $P$ which are not flexes such that there exist no flex $\hat{R}$ with $P \in T_{\hat{R}}C$.  
Let $\hat{P} \in \Delta \setminus r^{-1}(\bigcup_{\hat{R} \in F(\hat{C})}T_{\hat{R}}C)$ and $P=r(\hat{P})$. 
It follows from Riemann-Hurwitz formula that $\hat{\pi}_P$ is ramified at some point $\hat{R} \in \hat{C}$. 
It follows from Lemma \ref{index} and Fact \ref{Galois covering}(2) that $P \in T_{\hat{R}}C$ and the order of $T_{\hat{R}}C$ at $\hat{P}_0$ is equal to $2$, for each $\hat{P} \in \Delta \setminus r^{-1}(\bigcup_{\hat{R} \in F(\hat{C})}T_{\hat{R}}C)$, a ramification point $\hat{R}$ and $\hat{P}_0 \in r^{-1}(C \cap T_RC)$ with $\hat{P}_0 \ne \hat{P}$. 
Therefore, $d-1$ should be even, by Fact \ref{Galois covering}(3). 
Let $n(P)$ be the number of multiple tangent lines for such a Galois point $P$. 

(II-2-1) Assume that $p \ne 2$. 
Then, $s(\gamma)=1$ and $q(\gamma)=1$. 
Since $\hat{\pi}_P$ has only tame ramifications, it follows from Riemann-Hurwitz formula that 
$$ 2g-2=-2(d-1)+\frac{d-1}{2} \times n(P). $$ 
By Lemma \ref{twoGalois}, we have at least $\mu(C) \times n(P)$ multiple tangents. 
It follows from Fact \ref{plucker} that 
$$ \mu(C) \times n(P) \le \frac{(d_0-1)(d_0-2)}{2}, $$
where $d_0=2g-2+2d$. 
Since $\mu(C)\ge d^3-6d^2+9d$ and $d_0 \le d^2-d$, we have an inequality 
$$ (d^3-6d^2+9d) \times \frac{2}{d-1} \le \frac{(d_0-1)(d_0-2)}{2(d_0-2)} \le \frac{d^2-d-1}{2}. $$
Then, we have $g_2(d):=3d^3-22d^2+36d-1\le 0$. 
Since $g_2(5)=4$ and $d-1$ is even, this is a contradiction.

(II-2-2) Assume that $p=2$. 
Then, $q(\gamma)=2$. 
If $s(\gamma) \ge 2$, then the number of tangent lines whose contact points are strictly less than $s(\gamma)$ is at most 
$$ 2g_C-2-s(\gamma)(2g_{C^*}-2) \le 2g_C-2+(d/2)\times 2<d^2-2d $$
by Riemann-Hurwitz formula. 
Since any tangent line contains at most one inner Galois point by Lemma \ref{twoGalois} and $\mu(C)-(d^2-2d)\ge d^3-7d^2+11d >0$ if $d \ge 5$, there exist an inner Galois point $P$ and a point $\hat{R} \in \hat{C}$ with $R =r(\hat{R}) \ne P$ such that $T_PC=T_{\hat{R}}C$ and $I_P(C, T_PC)=\nu_{\hat{R}}=2$. 
Considering the projection $\hat{\pi}_P$ and Lemma \ref{index} and Fact \ref{Galois covering}(2), this is a contradiction.  
Therefore, $s(\gamma)=1$.

By the above arguments, we have the Theorem and the following in $p=2$. 

\begin{proposition}
Assume that $p=2$, $\delta(C) \ge (d-1)^3+1$ and $C$ is not projectively equivalent to a Hermitian curve. 
Then, we have the following.
\begin{itemize}
\item[(1)] $d$ is odd, $M(C)=q(\gamma)=2$ and $s(\gamma)=1$. 
\item[(2)] There exists no singular point with multiplicity $d-1$.   
\item[(3)] For any point $\hat{R} \in \hat{C}$, there exist a line in $\Bbb P^2$ defined by $h=0$ around $r(\hat{R})$ such that ${\rm ord}_{\hat{R}}r^*h=1$.  
\end{itemize}
\end{proposition}

\begin{center} {\bf Acknowledgements} \end{center} 
The author was partially supported by Grant-in-Aid for Young Scientists (B) (22740001), JSPS, Japan.

\end{document}